\documentclass[12pt]{article}
\usepackage[centertags]{amsmath}
\usepackage{amsfonts}
\usepackage{amssymb}
\usepackage{amsthm}
\usepackage{tikz-cd}
\input{amssym.def}

\newtheorem{Definition}{Definition}[section]
\newtheorem{Theorem}[Definition]{Theorem}
\newtheorem{Lemma}[Definition]{Lemma}
\newtheorem{Proposition}[Definition]{Proposition}

\newtheorem{Example}[Definition]{Example}

\newcommand{\lc}{\mathcal{L}}
\newcommand{\rc}{\mathcal{R}}
\newcommand{\hc}{\mathcal{H}}
\newcommand{\jc}{\mathcal{J}}

\title{\Large \bf Idempotent ordered semigroup}
\author{ K. Hansda \\
\footnotesize{Department of Mathematics, Visva-Bharati
University,}\\
\footnotesize{Santiniketan, Bolpur - 731235, West Bengal, India}\\
\footnotesize{kalyanh4@gmail.com}}

\begin{document}

\maketitle

\begin{abstract}{\footnotesize}
An element $e$  of an ordered semigroup $(S, \cdot, \leq )$ is
called an ordered idempotent if $e \leq e^2$. We call an ordered
semigroup $S$ idempotent ordered semigroup if every element of $S$
is an ordered idempotent.  Every idempotent semigroup is a complete
semilattice of  rectangular idempotent semigroups and in this way we
arrive to many other important classes of idempotent ordered
semigroups.

\end{abstract}

\section{Introduction}
Idempotents play an important role in  different major subclasses of
a regular semigroup. A regular semigroup $S$ is called orthodox if
the set of all idempotents $E(S)$ forms a subsemigroup, and $S$ is a
band if $S=E(S)$.

T. Saito studied systematically the influence of order on idempotent
semigroup \cite{Saito1962}.  In \cite{bh1}, Bhuniya and Hansda
introduced the notion of ordered idempotents and  studied different
classes of regular ordered semigroups, such as, completely regular,
Clifford and left Clifford  ordered semigroups by their ordered
idempotents.   The purpose of this paper to study  ordered
semigroups in which every element is  an ordered idempotent.
Complete semilattice decomposition of these semigroups automatically
suggests the looks of rectangular idempotent semigroups and in this
way we arrive to many other important classes of idempotent ordered
semigroups.

The presentation of the article is as follows. This section is
followed by preliminaries. In Section 3, is devoted to the
idempotent ordered semigroups and characterizations of different
type of idempotent ordered semigroups.

\section{Preliminaries}
An ordered semigroup  is a partially ordered set $(S, \leq)$, and at
the same time a semigroup $(S,\cdot)$ such that $ \textrm{for all}
\; a, b , c \in S;  \;a \leq b \;\textrm{implies} \;ca \leq cb
\;\textrm{and} \;ac \leq bc$. It is denoted by $(S,\cdot, \leq)$.
Throughout this article, unless stated otherwise, $S$ stands for an
ordered semigroup. For every subset $H\subseteq S$, denote $(H] =
\{t \in S : t \leq  h, \;\textrm{for \;some} \;h \in H \}$. Let $I$
be a nonempty subset of an ordered semigroup $S$. $I$ is a left
(right) ideal of $S$, if $SI \subseteq I \;( I S \subseteq I)$ and
$(I]= I$. $I$ is an ideal of $S$ if $I$ is both a left and a right
ideal of $S$. An (left, right) ideal $I \;\textrm{of} \;S$ is proper
if $I \neq S$. $S$ is left (right) simple if it does not contain
proper left (right) ideals. $S$ is called t-simple if it is both
left and right simple and  $S$ is  called simple if it does not
contain any proper ideals.

Kehayopulu \cite{ke91} defined Green's relations on a regular
ordered semigroup $S$ as follows:
$$ a \lc b   \; if   \;(S^1a]=(S^1b],  \;a \rc b   \; if   \;(aS^1]=(bS^1], \;a \jc b   \; if   \;(S^1aS^1]=(S^1bS^1], \;\textrm{and} \;\hc= \;\lc \cap
\;\rc.$$ These four relations $\lc, \;\rc, \;\jc \;\textrm{and}
\;\hc$ are equivalence relations.

An equivalence relation $\rho$ on $S$ is called left (right)
congruence if for every $a, b, c \in S; \;a \;\rho \;b \;
\textrm{implies}$ $ \;ca \;\rho \;cb \;(ac \;\rho \;bc)$. By a
congruence we mean both left and right congruence. A congruence
$\rho$ is called a semilattice congruence on $S$ if for all $a, b
\in S, \;a \;\rho \;a^{2} \;\textrm{and} \;ab \;\rho \;ba$. By a
complete semilattice congruence we mean a semilattice congruence
$\sigma$ on $S$ such that for $a, b \in S, \;a \leq b$ implies that
$a \sigma ab$.  Equivalently: There exists a semilattice $Y$ and a
family of subsemigroups $\{S\}_{\alpha \in Y}$ of type $\tau$ of $S$
such that:
\begin{enumerate}
\item \vspace{-.4cm}
$S_{\alpha}\cap S_{\beta}= \;\phi$ for any $\alpha, \;\beta \in \;Y
\;with \; \alpha \neq \beta,$
\item \vspace{-.4cm}
$S=\bigcup _{\alpha \;\in \;Y} \;S_{\alpha},$
\item \vspace{-.4cm}
$S_{\alpha}S_{\beta} \;\subseteq \;S_{\alpha \;\beta}$ for any
$\alpha, \;\beta \in \;Y,$
\item \vspace{-.4cm}
$S_{\beta}\cap (S_{\alpha}]\neq \phi$ implies $\beta \;\preceq
\;\alpha,$ where $\preceq$ is the order of the semilattice $Y$
defined by \\$\preceq:=\{(\alpha,\;\beta)\;\mid
\;\alpha=\alpha\;\beta\;(\beta\;\alpha)\}$ \cite{ke 2008}.
\end{enumerate}

An ordered semigroup $S$ is called left regular if for every $a \in
S, \;a \in (Sa^2]$. An element $e$ of  $(S, \cdot, \leq )$ is called
an ordered idempotent \cite{bh1} if $e \leq e^2$. As an immediate
example of idempotent ordered semigroups, we can consider
$(\mathbb{N}, \;\cdot, \;\leq )$ which is  an idempotent ordered
semigroup but $(\mathbb{N}, \;\cdot)$ is not an idempotent
semigroup. An ordered semigroup  $S$ is called $\hc-$commutative if
for every $a, b \in S, \;ab\in (bSa]$.

If $F$ is a semigroup, then the set $P_f(F)$ of all finite subsets
of $F$ is a semilattice ordered semigroup with respect to the
product $'\cdot'$ and partial order relation $'\leq'$ given by: for
$A, B \in P_f(F)$,
$$ A \cdot B = \{ab \mid a \in A, b \in B\} \;
\textrm{and} \; A \leq B \; \textrm{if and  only  if} \; A \subseteq
B.$$

\section{Idempotent ordered semigroups}
We have discussed the elementary properties of ordered idempotents.
In this section we characterize ordered semigroups of which every
element is an ordered idempotent. Here we show that  these ordered
semigroups are analogous to bands.

We now  make a natural analogy between band and  idempotent ordered
semigroup.
\begin{Theorem}
Let $B$ be a semigroup. Then $P_f(B)$  is idempotent ordered
semigroup if and only if $B$ is a band.
\end{Theorem}
\begin{proof}
Let $B$ be a band and $U\in P_f(B)$. Choose $x \in U$. Then $x^2 \in
U^2$ implies $x \in U^2$. Then $U \subseteq U^2$. So $P_f(B)$ is
idempotent ordered semigroup.

Conversely, assume that $B$ be a semigroup such that $P_f(B)$ is an
idempotent  ordered semigroup. Take $y \in B$. Then $Y=\{y\}\in
P_f(B)$. Thus $Y\subseteq Y^2$, which implies $y=y^2$. Hence $B$ is
a band.
\end{proof}
\begin{Proposition}
Let $B$ be a band, $S$ be an idempotent  ordered semigroup and $f:B
\longrightarrow S$ be a semigroup homomorphism. Then there is an
ordered semigroup homomorphism $\phi : P_f(B) \longrightarrow S$
such that the following diagram is commutative:
\begin{center}
\begin{tikzcd}[row sep=3.5em, column sep=3.5em]
B \arrow{r}{f} \arrow{d}{l} &S \\
P_f(B) \arrow[swap]{ur}{\phi}
\end{tikzcd}
\end{center}
where $l : B\longrightarrow P_f (B)$ is given by $l(x) = \{x\}$.
\end{Proposition}
\begin{proof}
Define $\phi : P_f(F) \longrightarrow S$ by: for $A \in P_f(F)$, $
\phi(A)=\vee_{a \in A} f(a)$. Then for every $A, B \in P_f(F),
\;\phi(AB)= \vee_{a \in A, b \in B}f(ab)= \vee_{a \in A, b \in
B}f(a)f(b)= (\vee_{a \in A}f(a))(\vee_{b \in B}f(b))= \phi(A)
\phi(B),$ and if $A \leq B$, then $\phi(A) = \vee_{a \in A}f(a) \leq
\vee_{b \in B}f(b) = \phi(B)$ shows that $\phi$ is an ordered
semigroup homomorphism. Also $\phi \circ l = f$.
\end{proof}
\begin{Lemma}\label{cr17}
In an idempotent  ordered semigroup $S, \; a^{m}\leq a^{n} \; for
\;every \; a\in S \;and \;m, n\in \mathbb{N}  \; with \; m\leq n $.
\end{Lemma}

Every idempotent ordered  semigroup $S$ is completely regular and
hence $\jc$ is the least complete semilattice congruence on $S$,  by
Lemma \ref{cr14}. In an idempotent ordered semigroup $S$, the
Green's relation $\jc$ can equivalently be expressed as:
$\textrm{for} \;a, b \in S$, $$a \jc b \;\textrm{if there are} \;x,
y, u, v \in S \;\textrm{such that} \;a \leq axbya \;\textrm{and} \;b
\leq buavb.$$

Now we characterize the $\jc-$class  in an idempotent ordered
semigroup.
\begin{Definition}
An idempotent ordered semigroup $S$ is called rectangular if for all
$a,b \in S$, there are $ x,y \in S $ such that $a\leq axbya $.
\end{Definition}
\begin{Example}
$(\mathbb{N} ,\cdot, \leq)$ is a rectangular idempotent ordered
semigroup,whereas if we define  $a\;o \;b= \;min \{a,b \} \;for
\;all \;a,b \in \mathbb{N}$ then $(\mathbb{N},\circ, \leq)$ is an
idempotent ordered semigroup but not rectangular.
\end{Example}
Also we have the following equivalent conditions.
\begin{Lemma}\label{cr18} Let $S$ be an idempotent ordered
semigroup. Then the following conditions are equivalent:
\begin{enumerate}
  \item \vspace{-.4cm}
$S$ is  rectangular;
\item \vspace{-.4cm}
for all $a,b \in S$, there is  $x \in S$ such that $a \leq axbxa$;
\item \vspace{-.4cm}
for all $a, b, c \in S$ there is  $x \in S$ such that $ac \leq
axbxc$.
\end{enumerate}
\end{Lemma}
\begin{proof}
$(1) \Rightarrow (3)$: Let $a,b,c \in S$. Then there are $x, y \in S
$ such that $a \leq axbya$. This implies that $ ac  \leq axbyac \leq
ax(bya)(bya)c \leq (axbyab)(axbyab)yac \leq a(axbyabya)b(axbyabya)c
\leq atbtc , \;\textrm{where} \;t=axbyabya \in S$.

$(3)\Rightarrow (2)$: Let $a,b \in S$. Then there is $x \in S$ such
that $a^{2} \leq axbxa$. Then  $a \leq a^2$ implies that $a \leq
axbxa$.

$(2)\Rightarrow (1)$: This follows directly.
\end{proof}
As we can expect, we show that the equivalence classes in an
idempotent ordered semigroup $S$ determined by $\jc$ are
rectangular.
\begin{Theorem}\label{cr19}
Every idempotent ordered semigroup  is a complete semilattice of
rectangular idempotent ordered semigroups.
\end{Theorem}
\begin{proof}
Let $S$ be an idempotent ordered semigroup. Then $\jc$ is the least
complete semilattice congruence on $S$. Now consider a $\jc$-class
$(c)_{\jc}$  for some $c \in S$. Since $\jc$ is a complete
semilattice congruence, $(c)_{\jc}$ is a subsemigroup of $S$. Let
$a, b \in (c)_{\jc}$. Then there is $x \in S$ such that $a \leq
axbxa $, which implies that $a \leq a(axb)b(bxa)a$, that is,  $a
\leq aubva$ where $u=axb \;\textrm{and} \;v= bxa$. Also  the
completeness of $\jc$ implies that $ (a)_{\jc}  = (a^2xbxa)_{\jc} =
(axb)_{\jc} = (bxa)_{\jc} $, and  $u, v \in (c)_{\jc}$. Thus
$(c)_{\jc}$ is a rectangular idempotent ordered semigroup.
\end{proof}
\begin{Definition}
An  idempotent ordered semigroup  $S$ is called left zero  if for
every $a,b \in S$, there exists $x \in S$ such that  $a \leq axb $.
\end{Definition}
An  idempotent ordered semigroup $S$ is called right zero  if for
every $a,b \in S$, there is $x \in S$ such that $a \leq bxa $.

Now it is evident that every left zero  and right zero idempotent
ordered semigroup is rectangular.

\begin{Proposition}\label{cr20}
An  idempotent ordered semigroup is left zero if and only if it is
left simple.
\end{Proposition}
\begin{proof}
First suppose that  $S$ is  a left zero idempotent ordered semigroup
and $a \in S$. Then for any $b \in S$, there is $x \in S$ such that
$b \leq bxa$, which shows that $b \in (Sa]$. Thus $S=(Sa] $ and
hence $S$ is left simple.

Conversely, assume that $S$ is left simple. So for every $a,b \in
S$, there is $s \in S$ such that $a \leq sb$. Then $a \leq a^2$
gives that $a^2 \leq asb$. Thus $S$ is a left zero idempotent
ordered semigroup.
\end{proof}

\begin{Lemma}\label{cr21}
In an idempotent ordered semigroup $S$, the following conditions are
equivalent:
\begin{enumerate}
\item \vspace{-.4cm}
For all $a, b \in S$, there is $x\in S$ such that  $ab \leq ab x
ba$;
\item \vspace{-.4cm}
For all $a, b \in S$, there is $x\in S$ such that  $ab \leq axbxa$;
\item \vspace{-.4cm}
For all $a, b \in S$, there is $x, y\in S$ such that  $ab \leq
axbya$.
\end{enumerate}
\end{Lemma}
\begin{proof}
$(1)\Rightarrow (3)$: This follows directly.

$(3)\Rightarrow (2)$: This is similar to the Lemma \ref{cr18}.

$(2)\Rightarrow (1)$: Let  $a, b \in S$. Then there is $x \in S$
such that $bab   \leq   baxbxba$. Now since  $ab \leq ababab,
\;\textrm{ we have} \;ab \leq ab(abaxbx)ba$.
\end{proof}
 We now introduce the notion of left regularity in an idempotent
 ordered semigroup.
\begin{Definition}
An idempotent ordered semigroup  $S$ is called left regular if for
every $a, b \in S$ there is $x \in S$ such that $ ab \leq axbxa $.
\end{Definition}

\begin{Theorem}\label{cr22}
An  idempotent ordered semigroup $S$ is left regular if and only if
$\lc= \jc$ is the least complete semilattice congruence on $S$.
\end{Theorem}
\begin{proof}
First we assume that $S$ is left regular. Let $a, b \in S$ be such
that $a \jc b$. Then there are $u, v, x, y \in S $ such that $a \leq
ubv \;\textrm{and} \;b \leq xay$. Since  $S$ is left regular, there
are $s, t \in S$ such that $bv \leq bsvsb \;\textrm{and} \;ay \leq
atyta$. Then $a \leq ubsvsb \;\textrm{and} \;b \leq xatyta $; which
shows that $a \lc b$. Thus $\jc \subseteq \lc$. Again $\lc \subseteq
\jc$ on every ordered semigroup and hence $\lc = \jc$. Since every
idempotent ordered semigroup is completely regular, it follows that
$\lc$ is the least complete semilattice congruence on $S$, by Lemma
\ref{cr14}

Conversely, let $\lc$ is the least complete semilattice congruence
on $S$. Consider $a, b \in S$. Then $ab \lc ba$ implies   that $ab
\leq xba$ for some $x \in S$. This implies that $$ab \leq abab \leq
abxba.$$ Thus $S$ is a left regular idempotent ordered semigroup, by
Lemma \ref{cr21}.
\end{proof}
In the following result a left regular idempotent ordered semigroup
has been characterized by left zero idempotent ordered semigroup.
\begin{Theorem}\label{cr23}
Let  $S$ be an idempotent ordered semigroup. Then the following
conditions are equivalent:
\begin{enumerate}
\item \vspace{-.4cm}
$S$ is left regular;
\item \vspace{-.4cm}
$S$ is a complete semilattice  of left zero idempotent ordered
semigroups ;
\item \vspace{-.4cm}
$S$ is a semilattice  of left zero idempotent ordered semigroups.
\end{enumerate}
\end{Theorem}
\begin{proof}

$(1)\Rightarrow (2)$: In view of Theorem \ref{cr22}, it is
sufficient to show that each $\lc$-class is a left zero idempotent
ordered semigroup. Let $ L$ be an $\lc$-class  and $a, b \in L$.
Then $L$ is a subsemigroup, since $\lc$ is a semilattice congruence.
Since $a \lc b$  there is $x \in S$ such that $a \leq xb$. This
implies that  $a \leq a^3 \leq a^2 xb \leq a^2 xb^2 \leq aub,
\;\textrm{where} \;u=axb$.

By the completeness of $\lc, \;a \leq xb$ implies that $(a)_{\lc}=
(axb)_{\lc}, \;\textrm{and hence} \;u \in L$. Thus $S$ is left zero
idempotent ordered semigroup.

$(2)\Rightarrow(3)$: Trivial.

$(3)\Rightarrow(1)$: Let $\rho$ be a semilattice congruence on $S$
such that each $\rho$-class is a left zero idempotent ordered
semigroup. Consider $a, b \in S$. Then $ab, \;ba \in (ab)_{\rho}$
shows that there is $s \in (ab)_{\rho}$ such that $ab \leq ab s ba
\leq a bsbs ba \leq a (bsb) b (bsb) a$. Hence $S$ is left regular.
\end{proof}
 The following theorem gives several alternative
characterizations of an $\hc-$commutative idempotent ordered
semigroup.
\begin{Lemma}\label{cr24}
Let $S$ be an idempotent ordered semigroup. Then the following
conditions are equivalent:
\begin{enumerate}
\item \vspace{-.4cm}
$S$ is $\hc$-commutative;
\item \vspace{-.4cm}
for all $a,b\in S, \;ab \in (baS] \cap (Sba];$
\item \vspace{-.4cm}
 $S$ is a complete semilattice of t-simple idempotent ordered
 semigroups;
\item \vspace{-.4cm}
 $S$ is a semilattice of t-simple idempotent ordered
 semigroups.
\end{enumerate}
\end{Lemma}
\begin{proof}
$(1)\Rightarrow (2)$: Consider  $a,b\in S$. Since $S$ is $\hc-$
commutative,  there is\\
$u \in S \;\textrm{such that} \;ab \leq bua $. Also  for $u,a \in S,
\;ua \leq asu$ for some $s \in S$. Thus $ab \leq basu$, which shows
that $ab \in (baS]$. Similarly $ab \in (Sba]$. Hence $ab \in
(baS]\cap(Sba]$.

$(2)\Rightarrow (3)$: Suppose that $J$ be an $ \jc$-class in $S$ and
$a, b \in J$. Since $J$ is rectangular there is $x \in J$ such that
$a \leq axbxa$. Also by the given condition $(2)$ there is $u \in J$
such that $bxa \leq xaub$. So $a \leq ax^2aub \leq vb
\;\textrm{where} \;v= ax^2au$. Since $\jc$ is a complete semilattice
congruence on $S,
 \;(a)_{\jc}=(a^2x^2aub)_{\jc}= (ax^2au)_{\jc}= (v)_{\jc}$.
So $v \in J$. This shows that $J$ is left simple. Similarly it can
be shown that $J$ is also right simple. Thus $S$ is a complete
semilattice of t-simple idempotent ordered semigroups.

$(3)\Rightarrow (4)$: This follows trivially.

$(4)\Rightarrow (1)$: Let  $S$ be the   semilattice $Y$ of t-simple
idempotent ordered semigroups $\{S_\alpha\}_{\alpha \in Y}$ and
$\rho$ be the corresponding semilattice congruence on $S$. Then
there are $\alpha, \beta \in Y$ such that $a \in S_\alpha
\;\textrm{and} \;b \in S_\beta $. Then $ba, ab \in S_{\alpha
\beta}$. Since $S_{\alpha \beta}$ is t-simple, $ab \leq xba$ for
some $x \in S_{\alpha \beta}$. Now for $x, ba \in S_{\alpha \beta}$
there is $y \in S_{\alpha \beta}$ such that $x \leq bay$. This
finally gives $ab \leq bta$,  where $t = ayb$.
\end{proof}

\begin{Definition}
An  idempotent ordered semigroup $(S,., \leq )$ is called weakly
commutative  if for any $a,b \in S \;\textrm{there exists} \;u \in S
\; \textrm{such \;that} \; ab\leq bua $.
\end{Definition}
\begin{Theorem}
For an idempotent ordered semigroup $S$, the followings are
equivalent:
\begin{enumerate}
  \item \vspace{-.4cm}
$S$ is weakly commutative;
  \item \vspace{-.4cm}
for any $a,b\in S, \;ab \in (baS]\cap(sba];$
  \item \vspace{-.4cm}
 $S$ is complete semilattice of left and right simple idempotent ordered semigroups .
\end{enumerate}
\end{Theorem}
\begin{proof}
$(1)\Rightarrow (2)$: Let   $a,b\in S$. Then $\exists u \in S \;s.t
\;ab \leq bua \;\textrm{ for} \;u,a \in S,
 \;\textrm{there exists} \;z \in S \;\textrm{such that}  \;ua \leq azu$.
Thus  $ab \leq bua \leq  baza \;for \;za \in S$. So $ab \leq (baS]$.
Similarly $ab\in (Sba]$.  Hence $ab \in (baS]\cap(sba]$.

$(2)\Rightarrow (3)$:  Since $S$ is an idempotent ordered semigroup,
we have by theorem[] $\rho$ is complete semilattice congruence. We
now have to show that, for each $z\in S, J=(z)_{\rho}$ is left and
right simple. For this let us choose $a,b\in J$. Then there exists $
x,y \in S \;\textrm{such \;that} \; a \leq axbya$. So from the given
condition $bya\in (syab]$ and therefore there is  $ \;s_{1} \in S
\;\textrm{such that} \;bya \leq s_{1}yab$. Therefore $a \leq
axs_{1}yab$. Now since $\rho$ is complete semilattice congruence on
$S$, we have  $(a)_{\rho}  = (a^{2}xs_{1}yab)_{\rho} =
(axs_{1}yab)_{\rho} = (axs_{1}ya)_{\rho}$. Thus $a \leq ub, \;
\textrm{where} \;u=axs_{1}ya \in J$. Hence $J$ is left simple and
similarly it is  right simple.

$ (3)\Rightarrow (1)$: Let $S$ is complete semilattice $Y$ of left
and right simple idempotent ordered semigroups
$\{S_{\alpha}\}_{\alpha \in Y}$. Thus $S= \{S_{\alpha}\}_{\alpha \in
Y}$. Take $a,b \in S$. Then there are  $ \alpha, \beta \in Y
\;\textrm{ such \;that} \;a\in S_{\alpha}\;\textrm{and} \;b \in
S_{\beta}$. Thus $ab \in S_{\alpha \beta} $. So $ab, ba \leq
S_{\alpha \beta}$. Then there are $u,v \in S_{\alpha \beta}
\;\textrm{such that}  \;ab \leq uba \;\textrm{and} \;ab \leq bav
\;\textrm{implies} \;ab\leq ab^{2}\leq bta, \; \textrm{where} \;t=
avub$.  Hence $S$ is weakly commutative. This completes the proof.

\end{proof}

\begin{Definition}
An idempotent ordered semigroup $(S, \cdot, \leq )$ is called normal
if for any $a,b,c \in S, \; \exists \;x \in S \; \textrm{such
\;that} \; abca \leq acxba $.
\end{Definition}
\begin{Theorem}
For an idempotent ordered semigroup $S$, the followings are
equivalent:
\begin{enumerate}
  \item \vspace{-.4cm}
$S$ is normal;
  \item \vspace{-.4cm}
$aSb$ is weakly commutative,  for any $a,b\in S$;
  \item \vspace{-.4cm}
$aSa$ is weakly commutative, for any $a\in S$.
\end{enumerate}
\end{Theorem}
\begin{proof}
$(1)\Rightarrow (2)$:  Consider $axb,ayb \in aSb \; for \;x,y \in
S$. As $S$ is normal, $\exists u,v \in S $ such that $(axb)(ayb)
\leq   (axb)(ayb)(axb)(ayb) \leq aybuxba^{2}xbayb, \;for \;xba, yb
\in S \leq  (ayb)uxb(bay)v(a^{2}x)b, \;for \;a^{2}x, bay \in S\leq
(ayb)(uxb^{2}ayva)(axb)$.
 This implies
 $(axb)(ayb) \leq (ayb)t(axb) \leq (ayb)(ayb)t(axb)(axb),
 \;t= uxb^{2}ayva$ and thus $(axb)(ayb)\leq aybsaxb, \;where \;s= aybtaxb \in aSb$.
Thus $aSb$ is weakly commutative.

$(2)\Rightarrow (3)$: This  is obvious by taking $b=a$.

$(3)\Rightarrow (1)$: Let  $a,b,c \in S$. Then $abca, aca \in aSa$.
Since $aSa$ is weakly commutative.  Then there is $ s \in aSa$ such
that $(abca)aca  \leq acasaabca$. Now for $aba, abca \in aSa$, there is $t \in aSa$ such that \\
$abaabca \leq abcataba$. Thus $abca \leq (abca)(abca) \leq
abca^{2}ca^{2}bca \leq abca^{2}ca^{2}ba^{2}bca= (abcaaca)(abaabca)
\leq(acasa^{2}ca)(abcataba)\leq acuba; \;\textrm{where}
\;u=asa^{2}bca^{2}bcata \in S$. Hence $S$ is normal.
\end{proof}

\begin{Definition}
An idempotent ordered semigroup $(S,\cdot, \leq )$ is called left
normal (right normal) if for any $a,b,c \in S, \;\textrm{there
exists} \;x \in S$ such that $ abc \leq acxb \;(abc \leq bxac) $.
\end{Definition}

\begin{Theorem}
Let $S$ be a left normal idempotent ordered semigroup, then
\begin{enumerate}
  \item \vspace{-.4cm}
${\lc}$ is  the least complete semilattice congruence on $S$;
  \item \vspace{-.4cm}
$S$ is a complete semilattice of LZidempotent ordered semigroups.

\end{enumerate}
\end{Theorem}
\begin{proof}
$(1)$: Let  $a,b \in S \;\textrm{such \;that} \;a \rho b$. Then
there are $ x,y,u,v \in S $ such that
\begin{align}\label{LN}
 a \leq a(xbya), b\leq b(uavb).
\end{align}
Since $S$ is left normal, we have for $x,b,ya \in S, xbya \leq xyatb
\;\textrm{for \; some} \;t \in S$. Similarly  there is $ s \in S
\;\textrm{such \;that} \;uavb \leq uvbsa$.  So from \ref{LN}, $a
\leq (axyat)b \; \textrm{and} \;b \leq (buvbs)a$. Hence $a \lc b$.
Thus $\rho \subseteq \lc$.

Again, let $a,b \in S \;\textrm{such \;that} \;a \lc b$. Thus there
re $ u,v \in S \;\textrm{such \;that} \;a \leq ub \;and \;b \leq va
$. By lemma[], we have $a \leq a^{3}= aaa \leq auba \leq aubba \;
\textrm{for \;some} \;u,b \in S $. Therefore $a \rho b$. Thus $\lc
\subseteq \rho$. Thus $\lc= \rho$.

$(2)$: Here we are only to proof that each $\lc$-class is a left
zero. For this let  $\lc$-class $(x)_{\lc}=L, \;(\textrm{say}) \;
\textrm{for \;some} \;x \in S$. Clearly $(x)_{\lc}$ is a
subsemigroup of $S$. Take $a,b \in L$. Then $y,z \in S
\;\textrm{such \; that} \;a \leq yb, b \leq za$. Since $S$ is left
normal, there is   $ t \in S \; \textrm{such \;that} \; a \leq yb
\leq (yb)b \leq yzab $.

This implies $a \leq a^{2} \leq a(ayzb)b$. Thus $(a)_{\lc}=
(a^{2}yzb)_{\lc}= (ayzb)_{\lc}$. Therefore  $L$ is left zero.
 Hence $S$ is a complete semilattice of left zero idempotent ordered semigroups.
\end{proof}

\begin{Theorem}
Let $S$ be a idempotent ordered semigroup, then $S$ is normal if and
only if $\lc$ is right normal band congruence and  $\rc$ is left
normal band congruence.
\end{Theorem}

\begin{proof}
First we shall see that  $\lc$ is left congruence on $S$. For this
let us take $a,b \in S \; \textrm{such \;that} \;a \lc b
\;\textrm{and} \;c \in S$. Then there is $x,y \in S \; \textrm{such
\;that} \;a \leq xb, b \leq ya$. Now as $S$ is normal idempotent
ordered semigroup, $ca \leq cxb \leq cxbcxb \leq cxbx(s_{1})cb
\;\textrm{for \;some} \;s_{1} \in S$. Thus $a \leq s_{2}cb
\;\textrm{where} \;s_{2}=cxbxs_{1} \in S$.  Again $cb \leq s_{4}ca
\;where \;s_{4}= cyays_{3} \in S$. So $ca \lc cb$. It finally  shows
that $\lc$ is congruence on $S$. Similarly it can be shown that $\rc$ is congruence on $S$.\\
Next consider that $a,b,c \in S$ are arbitrary. Then since $S$ is
normal idempotent ordered semigroup, $abc \leq abcabc \leq
abcbt_{1}ac \leq acb(t_{1}t_{2}bac) \;\textrm{for \;some}
\;acbt_{1}t_{2} \in S. $ Also $bac \leq bacbac \leq bacat_{3}bc \leq
(bct_{3}t_{4}abc) \;\textrm{for \;some} \;bct_{3}t_{4} \in S$. So
$abc \lc bac$. Similarly $abc \rc acb$. This two relations
respectively shows that  $\lc$ is right normal band congruence and
$\rc$ is left normal band congruence.

Conversely, suppose that  $\lc$ is right normal band congruence and
$\rc$ is left normal band  congruence. Consider $a,b,c \in S$. Then
$abc \rc acb \; and \;bca \lc
 cba$. Then $\exists x_{1},x_{2} \in S$ such that
 \begin{align*}
 abc \leq (acb)x_{1} \;and \;bca \leq x_{2}cba.
\end{align*}
Now then $abc \leq (abc)bca \leq (acb)x_{1}bca \leq
ac(bx_{1}x_{2}c)ba \;for \;some \;bx_{1}x_{2}c \in
 S$. Hence $S$ is a Nidempotent ordered semigroup.

\end{proof}

\bibliographystyle{plain}

\end{document}